\documentclass[11pt]{article}
\usepackage{amssymb, amsmath, amsthm, epsf, epsfig}

\newcommand{\COLORON}{1}
\newcommand{\NOTESON}{0}
\newcommand{\Debug}{0}
\usepackage[usenames,dvips]{color} 
\usepackage{amsthm,amssymb,amsmath,enumerate,graphicx,epsf,mathbbol,stmaryrd}
\usepackage[dvips, bookmarks, colorlinks=false, breaklinks=true]{hyperref}

\newcommand{\comment}[1]{}
\newcommand{\COMMENT}[1]{}

\definecolor{darkgray}{rgb}{0.3,0.3,0.3}

\comment{
	\begin{lemma}\label{}	
\end{lemma}
\begin{proof}

\end{proof}

\begin{theorem}\label{}
\end{theorem} 
\begin{proof} 	

\end{proof}

}

\newtheorem{proposition}{Proposition}[section]

\newtheorem{theorem}[proposition]{Theorem}

\newtheorem{lemma}[proposition]{Lemma}

\newtheorem{examp}[proposition]{Example}

\newcommand{\FIG}{0}

\ifnum \NOTESON = 1 \newcommand{\note}[1]{ 

	\ 

	{\color{blue} \hspace*{-60pt} NOTE: \color{Turquoise}{\small  \tt \begin{minipage}[c]{1.1\textwidth}  #1 \end{minipage} \ignorespacesafterend }} 
	
	\ 
	
	}
\else \newcommand{\note}[1]{} \fi

\newcommand{\afsubm}[1]{ \ifnum \Debug = 1 {\mymargin{#1}}
\fi} 

\ifnum \Debug = 1 
\else  \fi

\ifnum \FIG = 1 
\else  \fi

\ifnum \FIG = 1 
\else  \fi

\ifnum \Debug = 1 \usepackage[notref,notcite]{showkeys}
\fi

\ifnum \COLORON = 0 \renewcommand{\color}[1]{}
\fi

\newcommand{\R}{\ensuremath{\mathbb R}}

\newcommand{\ce}{\ensuremath{\mathcal E}}

\newcommand{\sig}{\ensuremath{\sigma}}

\newcommand{\g}{\ensuremath{G\ }}
\newcommand{\G}{\ensuremath{G}}

\newcommand{\Ex}{\mathbb E}

\newcommand{\Lr}[1]{Lemma~\ref{#1}}
\newcommand{\Tr}[1]{Theorem~\ref{#1}}
\newcommand{\Sr}[1]{Section~\ref{#1}}

\renewcommand{\iff}{if and only if}

\newcommand{\Fe}{For every}

\newcommand{\st}{such that}

\newcommand{\wrt}{with respect to}

\newcommand{\labequ}[2]{ \begin{equation} \label{#1} #2 \end{equation} } 
\newcommand{\labtequ}[2]{ \begin{equation} \label{#1} 	\begin{minipage}[c]{0.9\textwidth}  #2 \end{minipage} \ignorespacesafterend \end{equation} }

\newcommand{\mymargin}[1]{
  \marginpar{%
    \begin{minipage}{\marginparwidth}\small%
      \begin{flushleft}%
        {\color{blue}#1}%
      \end{flushleft}%
   \end{minipage}%
  }%
}%

\newcommand{\mySection}[2]{}

\newlength{\originalbase}
\newcommand{\spacing}[1]{\setlength{\baselineskip}{#1\originalbase}}

\newcommand{\C}{{\rm C}}
\newcommand{\CR}{{\rm CR}}

\newcommand{\Res}{{\cal R}}
\newcommand{\Con}{{\cal C}}

\begin{document}
\spacing{1.5}

\newcommand{\comxy}{\ensuremath{{\mathbb E}T^{x \leftrightarrow y}}}
\newcommand{\effrxy}{\ensuremath{\Res^{x y}}}
\newcommand{\cmfwxy}{\ensuremath{T^{x y}_{A\rightarrow}}}
\newcommand{\cmfwyx}{\ensuremath{T^{x y}_{A\leftarrow}}}
\newcommand{\cminxy}{\ensuremath{T^{x  y}_{A}}}
\newcommand{\cmaxxy}{\ensuremath{T^{x y}_{A\leftrightarrow}}}
\newcommand{\Ecmfwxy}{\ensuremath{\Ex \cmfwxy}}
\newcommand{\Ecmfwyx}{\ensuremath{\Ex \cmfwyx}}
\newcommand{\Ecminxy}{\ensuremath{\Ex \cminxy}}
\newcommand{\Ecmaxxy}{\ensuremath{\Ex \cmaxxy}}
\newcommand{\andcom}[1]{\ensuremath{\overleftrightarrow{#1}}-commute}

\title{New Bounds for Edge-Cover by Random Walk}

\author{Agelos Georgakopoulos\thanks{Department of Mathematics,
Technische Universit\"{a}t Graz, Austria; georgakopoulos@tugraz.at. The first author acknowledges the kind hospitality of the Dartmouth College and the second author.}
\, and Peter Winkler\thanks{Department of Mathematics, Dartmouth,
Hanover NH 03755-3551, USA; peter.winkler@dartmouth.edu.  Research
supported by NSF grant DMS-0901475.}}

\maketitle

\begin{abstract}
We show that the expected time for a random walk on a (multi-)graph $G$ to traverse
all $m$ edges of $G$, and return to its starting point, is at most $2m^2$;
if each edge must be traversed in both directions, the bound is $3m^2$.
Both bounds are tight and may be applied to graphs with arbitrary edge lengths,
with implications for Brownian motion on a finite or infinite network of
total edge-length $m$.
\end{abstract}

\section{Introduction}
\subsection*{Overview}

The expected time for a random walk on a graph $G$ to hit all the vertices of $G$
has been extensively studied by probabilists, combinatorialists and computer scientists;
applications include the construction of universal traversal sequences \cite{AKLLR, Br},
testing graph connectivity \cite{AKLLR, KR}, and protocol testing \cite{MP}.
It has also been studied by physicists interested in the fractal structure of the uncovered
set of a finite grid; see \cite{DPRZ} for references and for an interesting relation between
the cover time of a finite grid and Brownian motion on Riemannian manifolds.

Here we consider the time to cover all {\em edges} of $G$, and moreover we take as the
fundamental parameter the number $m$ of edges of $G$, rather than the number of vertices.
Indeed, if the edges are taken to be of various lengths, then the number of vertices is no
longer of interest and the total edge-length (also denoted by $m$) becomes the natural
parameter by means of which to bound cover time.

\subsection*{Earlier Results}
Let $G= \langle V,E \rangle$ be a connected, undirected graph (possibly with loops and multiple edges),
where $|V|=n$ and $|E|=m$.  Let $r \in V$ be a distinguished ``root'' vertex.  A (simple)
{\em random walk} on $G$ beginning at $r$ is a Markov chain on $V$, defined as follows:
if the walk is currently at vertex $x$ it chooses an edge incident to $x$ uniformly at
random, then follows that edge to determine its next state.

The expected value of the least time such that all vertices in $V$ have been hit by the
random walk will be denoted by $\C^V_r(G)$ or simply $\C^V_r$, and we call this number the
(expected) {\em vertex cover time} of $G$ from $r$.  The vertex cover time $\C^V(G)$ of $G$ is defined
to be the maximum of $\C^V_r(G)$ over all $r \in V$.

Upper bounds on $\C^V(G)$ have been widely sought by the theoretical computer science community, and many results
obtained.  The cover time of a simple graph (no loops or multiple edges) was shown \cite{AKLLR}
to be at most cubic in $n$; more recently Feige \cite{F2} refined the bound to $\frac{4}{27} n^3$
plus lower-order terms, which is realized by the same ``lollipop'' graph that maximizes expected
hitting time \cite{BW1}---namely, a clique on $\frac23 n$ vertices attached to one end of a path
on the remaining vertices.

In fact \cite{AKLLR} shows that $\C^V(G)$ is O$(mn)$, thus trivially O$(m^2)$, even if the walk is
required to return to the root.  The expected time to cover all vertices and return to the
root is called the (vertex) {\em cover and return time}, here denoted by $\CR^V_r$; it is often easier
to work with than $\C^V_r$.  Among simple graphs (no loops or multiple edges) with $m$ edges,
the path $P_{m+1}$ on $m$ edges---starting in the middle---provides the greatest vertex cover time,
namely $m^2 + \lfloor m/2 \rfloor \lceil m/2 \rceil \sim \frac54 m^2$ \cite{F2}.
The path started at one end had earlier been shown \cite{BW2} to have the strictly
greatest vertex cover and return time among all trees with $m$ edges.

The expected time for a random walk beginning at $r$ to cover all the {\em edges} of $G$ will
be denoted by $\C^E_r(G)$; if the walk is required to cover all the {\em arcs}, that is,
to traverse every edge of $G$ in both directions, the expected time is denoted by $\C^A_r(G)$.
If the walk must return to $r$, we use $\CR^E_r$ and $\CR^A_r$ as above.  Edge-cover is for
some purposes more natural than vertex-cover; for example, it makes more sense on a graph
with loops and multiple edges, where there is no bound on any kind of cover based only
on the number $n$ of vertices, but there are natural, tight bounds (as we shall see) on
$\CR^E$ and $\CR^A$ in terms of the number of edges.

Bounds on expected edge-cover times were obtained by Bussian \cite{B} and Zuckerman \cite{Z1,Z2}.
Bussian noted that to cover the edges of a graph $G$, it suffices to hit each vertex somewhat
more often than its degree. Recently Ding, Lee and Peres \cite{DLP} proved a conjecture of
Winkler and Zuckerman \cite{WZ} saying roughly that the expected time to hit every vertex
{\em often} is no more than a constant times the vertex-cover time, which for a regular
graph is bounded by $6n^2$ \cite{KLNS}; it follows that the edge- (or arc-) cover time for
a regular graph is O$(mn)$, although the constant may be awful.  For general graphs,
Zuckerman \cite{Z2} succeeded in showing $\CR^A(G) \le 22 m^2$, using probabilistic methods.  

\subsection*{Our Results}

The main result of this paper is that if \g is a multigraph with $m$ edges then the expected
time it takes for a simple random walk started at any vertex to traverse all the edges of $G$
(and return to its starting point) is at most $2m^2$; in fact, we can even fix an orientation
of each edge beforehand, and $2m^2$ steps will be enough on average to traverse each edge in its
pre-specified direction.  If instead we insist that each edge be traversed both ways, then we
have to wait at most $3m^2$ steps on average.

As an intermediate step in our proof of the above bounds we obtain a refinement of a well-known formula of Chandra et.\ al.\ \cite{CRRST} concerning the commute time of random walk.

These results apply more generally if the edges of \g are assigned positive real lengths and
the transition probabilities of the random walk, as well as the time it takes to make a step,
are appropriately adapted.  In that case the above bounds apply when $m$ is the sum of the
lengths of the edges of $G$.

We call a graph with edge-lengths a {\em network}, and interpret it as an {\em electrical network}
by equating edge-length with resistance along an edge.  The usual connections between random walk
and electrical networks (as in \cite{DS,T} or see below) apply provided that in the definition of
random walk, the probability that an edge adjacent to the current position is traversed is
proportional to its conductance---that is, inversely proportional to its resistance (length). 
Such random walk has the same distribution as the sequence of vertices visited by standard Brownian motion.

To effect the correct scaling, and to make our results work for standard Brownian motion, we say that traversing
an edge of length $\ell$ takes time $\ell^2$. In fact the expected time for a standard Brownian particle to proceed from
one vertex to a neighboring vertex is {\em not} generally the square of the length of the edge;
that the ``$\ell^2$'' model for discrete random walk permits extending our results to Brownian
motion is a consequence of an averaging argument to be elucidated in Section~\ref{secBMf}.

For example, the expected time for a random walk
to proceed from one end to the other of a path $P_{m+1}$ of length $m$ (that is, a path consisting of
$m{+}1$ vertices and $m$ unit-length edges) is exactly $m^2$; this matches the time taken by
standard Brownian motion to travel from 0 to $m$ on the non-negative $X$-axis, and the time for our
generalized random walk to travel from one end to the other of any linear network, regardless
of the number and placement of vertices on it.  Thus it works in all senses to say that the
expected edge-cover time $\C^E_a(P_{m+1})$ of a path of length $m$ from an endpoint $a$ is $m^2$. 

Our bounds have interesting implications for Brownian motion on infinite networks. 
Motivated by earlier work of the first author \cite{G}, Georgakopoulos and Kolesko \cite{GK} study 
Brownian motion on infinite networks in which the total edge length $m$ is finite. Applied to this context, our results imply that Brownian motion will cover all edges in expected time at most $2m^2$, thus almost surely in finite time. See Section~\ref{secBMi} for more details.

\section{Definitions} \label{secD}

\subsection{Random Walk}

We will denote by $N=\langle G,\ell \rangle$ a network with underlying multigraph $G$ and
edge-lengths $\ell: E(G)\to \R^+$, also interpreted as resistances.  An {\em arc} of $N$
is an oriented edge; in our context, a loop comprises two arcs.  A {\em random walk} on
$N$ begins at some vertex $r$ and when at vertex $x$, traverses the arc $(x,y)$ to $y$
with probability 
\labequ{trpr}{
\frac{1/\ell(x,y)}{\sum_{z \sim x}1/\ell(x,z)}.
}

Our random walks take place in continuous time. The time it takes our random walker to perform a step,
i.e.\ to move from a vertex $x$ to one of its neighbors, depends on the lengths of the edges incident
with $x$. There are various ways to define this dependence. In this paper we will consider two natural
models for this dependence. Our main results will apply to both models with identical proofs, except
that it takes a different argument to prove \eqref{2lr} below.

\subsubsection*{The Brownian model}

Brownian motion on a line extends naturally to Brownian motion on a network (see, e.g., \cite{BC,BPY,FD,V}); when
at a vertex a Brownian particle moves with equal likelihood onto any incident edge (with the understanding that
an incident loop counts for this purpose as two incident edges).  The edges incident to a vertex constitute
a ``Walsh spider'' (see, e.g., \cite{Wa,BPY}) with equiprobable legs, and it is easily verified that
in such a setting the probability of traversing a particular incident edge (or oriented loop) first is proportional
to the reciprocal of the length of that edge.

In our model, we care only about where the particle is after an edge-traversal, and how long it took to
get there.  The latter should thus be the time taken by a Brownian particle to traverse a given edge incident
to its starting vertex, given that it traversed that edge first; this time is a random variable whose distribution
and expectation (the latter to be computed later) depend not only on the length of the specified edge, but on the
lengths of the other incident edges as well.

In the case of {\em simple random walk}, i.e.\ when all edges have length 1, the expected time for this random walk
to take a step is 1. 

\subsubsection*{The $\ell^2$ model}

In this model, the time it takes for the random walk to take a step is less random: traversing an edge of
length $\ell$ always takes time $\ell^2$. Thus time is governed by \eqref{trpr} alone.
The $\ell^2$ model and the Brownian model differ only in timing; probabilities of walks are identical, as in both cases the
next edge to be traversed is chosen according to the distribution in \eqref{trpr} for incident edges.
If all edges of the network are of the same length, then in expectation, timing is identical as well.
Readers interested only in simple random walk will lose nothing by assuming throughout that all
edges are of length 1.

\medskip

The expected time to cover all edges (respectively, arcs) of $N$ by a random walk (in either model)
will be denoted by $\C^E_r(N)$ (resp., $\C^A_r(N)$); to cover all edges or arcs
and return, by $\CR^E_r(N)$ or $\CR^A_r(N)$.  Maximizing over $r$ gives $\C^E(N)$, $\C^A(N)$,
$\CR^E(N)$, and $\CR^A(N)$. 

The {\em effective resistance} $\Res^{xy}$ between vertices $x$ and $y$ is defined in electrical
terms as the reciprocal of the amount of current that flows from $x$ to $y$ in $N$ when a unit
voltage difference is applied to them, assuming that each edge offers resistance equal to its length. See \cite{DS} or \cite{L} for the basic definitions concerning electrical networks.
Effective resistances sum in series: if all paths
from $x$ to $y$ go through $z$, then $\Res^{xy}= \Res^{xz} + \Res^{zy}$.  The reciprocals
of effective resistances, i.e., effective conductances, sum in parallel:  if $A$ and $B$ are
otherwise disjoint networks containing $x$ and $y$, then in the union of the two networks,
$$
\frac1{\Res^{xy}} = \frac1{\Res_A^{xy}} + \frac1{\Res_B^{xy}}.
$$

\section{Commute Times} \label{seccom}

The {\em commute time} $T^{x \leftrightarrow y}$ from vertex $x$ to $y$ in $N$ is the
(random) time it takes for a random walk to travel from $x$ to $y$ and back to $x$.
The expected commute time $\comxy$ between two vertices $x,y$ of a network has an elegant
expression proved in \cite{CRRST}, and well known for the case of unit edge lengths:
\labtequ{2lr}{$\comxy = 2 m \effrxy$,} where $m:= \sum_{e\in E(G)} \ell(e)$ is the total
length of the network and $\effrxy$ the effective resistance as defined above.

In fact a more general identity is proved in \cite{CRRST} (Theorem 2.2 there): 
suppose each traversal of an arc $\vec{e}$ comes with a cost $f(\vec{e})$, where $f$ is a possibly asymmetric cost function, and transition probabilities are still given by \eqref{trpr}. Then the expected cost of an $x$-$y$~commute is
$F \effrxy$, where $F$ is the sum of ${f(\vec{e})} / {\ell(e)}$ over all arcs $\vec{e}$.  Substituting
$\ell(e)^2$ for $f(\vec{e})$ gives the familiar $2 m \effrxy$
for commute time in the $\ell^2$ model, with $m$ now understood as the total length of the network.

That the same formula applies to the Brownian model follows by approximating edge-lengths with
rational numbers, then subdividing so that every edge has the same length; this has no effect
on commute time in the Brownian model but brings traversal times in line with the $\ell^2$ model
without changing the formula.

\medskip

In the rest of this section we refine \eqref{2lr} by differentiating
between various commute tours according to their behavior \wrt\ subnetworks. We will need this refinement in order to prove our main results in the next section.

Suppose that $N$ is the union of two subnetworks $A,B$ \st\ $A \cap B = \{x,y\}$
(as in Fig.~\ref{fig:1} below).
\begin{figure}[ht]
\epsfxsize240pt
$$\epsfbox{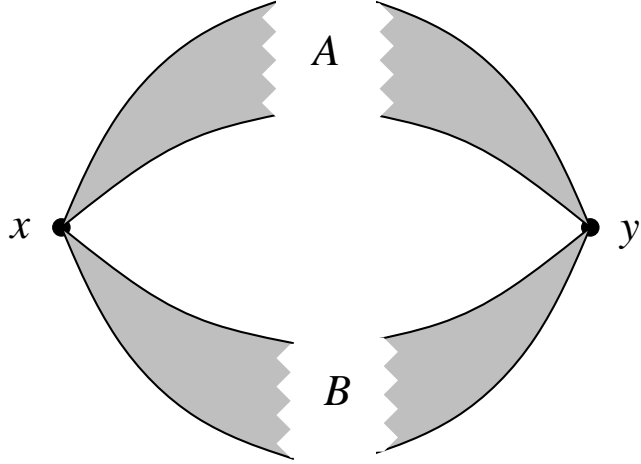}$$
\caption{Networks $A$ and $B$ meeting only at vertices $x$ and $y$.}\label{fig:1}
\end{figure}
For example, $A$ could be an $x$-$y$~edge and $B$ could be the rest of $N$; this is in fact the
case that is needed later.  We define the following events for random walks on $N$ starting at $x$:
\begin{enumerate}
\item An $A$-commute from $x$ to $y$ is a closed walk starting at $x$ and containing either
an $x$-$y$~subwalk via $A$ or a $y$-$x$~subwalk via $A$;
\item An $\overrightarrow{A}$-commute from $x$ to $y$ is a closed walk starting at $x$ and containing an $x$-$y$~subwalk via $A$;
\item An $\overleftarrow{A}$-commute from $x$ to $y$ is a closed walk starting at $x$ and containing a $y$-$x$~subwalk via $A$;
\item An \andcom{A}\ from $x$ to $y$ is a closed walk starting at $x$ and containing both an $x$-$y$~subwalk via $A$ and a $y$-$x$~subwalk via $A$.
\end{enumerate}
Define $\cminxy, \cmfwxy,\cmfwyx,\cmaxxy$ to be the time  for random walk on $N$ starting at $x$ to perform an $A$-commute,
a $\overrightarrow{A}$-commute, a $\overleftarrow{A}$-commute, and a \andcom{A}\ from $x$ to $y$, respectively.
Let $\Ecminxy, \Ecmfwxy,\Ecmfwyx,\Ecmaxxy$ denote the corresponding expected times.

\begin{theorem} \label{thcom} For every network $N=(G,\ell)$ and any two subnetworks $A,B$ having precisely two vertices $x,y$ in common,
\begin{enumerate}
\item \label{i}   $\Ecminxy = 2m \Res_A \frac{\Res_A + \Res_B}{2\Res_A + \Res_B}$;
\item \label{ii}  $\Ecmfwxy = 2m \Res_A \frac{2\Res_A + \Res_B}{2\Res_A + \Res_B} = 2m \Res_A =\Ecmfwyx$;
\item \label{iii} $\Ecmaxxy = 2m \Res_A\frac{3\Res_A + \Res_B}{2\Res_A + \Res_B}$,
\end{enumerate}
where $m:= \sum_{e\in E(G)} \ell(e)$ is the total edge-length, and $\Res_X$ is the effective resistance between $x$ and $y$
in the network $X$.
\end{theorem}
{\em Remark:} Notice that the three expressions differ only in the coefficient of $\Res_A$ in the numerator.
\begin{proof}
In all cases, let the particle start at $x$ and walk on $N$ `forever', and consider the times when the particle
completes a commute from $x$ to $y$. We are interested in the expectation of the time $T$ when the first $A$- commute,
$\overrightarrow{A}$-commute, or  \andcom{A}\ is completed according which of the cases \ref{i} - \ref{iii} we are considering.

Note that any $A$-commute, $\overrightarrow{A}$-commute, or  \andcom{A}\ from $x$ to $y$ can be decomposed into a sequence of disjoint commutes.
Thus we can define a random variable $Y$ to be the number of $x$-$y$ commutes until the first $A$-commute, $\overrightarrow{A}$-commute,
or \andcom{A}\ from $x$ to $y$ is completed, according to which of the cases \ref{i} - \ref{iii} we are considering.
Let $X_i, i=1,2,\ldots$ be the duration of the $i$th $x$-$y$ commute. Then our stopping time $T$ satisfies
$$
T= \sum_{1\leq i \leq Y} X_i.
$$
 
Note that the expectation of each $X_i$ is $\comxy = 2 m \effrxy$ by \eqref{2lr}.
By Wald's identity \cite{W} the expectation of $T$ equals the expectation of $Y$ times the expectation of $X_i$, and so we have 
\labtequ{wald}{$\Ex T = 2 m \effrxy \Ex Y$.}
Thus it only remains to determine $\Ex Y$ in each of the cases \ref{i} - \ref{iii}.

In order to compute $\Ex Y$ it is useful to first calculate the probability $p_A$ that the first visit to $y$ of a random walk starting at
$x$ will be via $A$.  For this it is convenient to use the electrical network technique.  Disconnect $A$ and $B$
at their common vertex $y$ to obtain a network consisting of $A,B$ connected `in series' at $x$.  Denote by $y_A$
the vertex of $A$ corresponding to $y$, and by $\Con_A = 1/\Res_A$ the effective conductance between $x$ and $y$ in $A$;
similarly for $B$.  From \cite{DS} we have the probability that a random walk started at $x$ hits $y_A$ before $y_B$
is $\Con_A/\Con_B$, thus $p_A = \frac{\Con_A}{\Con_A + \Con_B}=\frac{\effrxy}{\Res_A}$.

By the same argument, a random walk from $y$  to $x$ will go via $A$ with the same probability $p_A$. 
We can now determine $\Ex Y$ in each of the three cases. For case \ref{ii}, note that a commute from $x$ to $y$
is a $\overrightarrow{A}$-commute if the forward trip is via $A$, which occurs with probability $p_A$.
Thus we have $\Ex Y =: \Ex Y_{ii} = 1/p_A= \frac{R_A}{\effrxy}$, since the expected number of Bernoulli trials
until the first success is the reciprocal of the success probability of one trial.  Plugging this into \eqref{wald}
yields $\Ex T =\Ecmfwxy= 2 m \Res_A$ as claimed. By similar arguments this expression also equals $\Ecmfwyx$.

For case \ref{i}, note that a commute from $x$ to $y$ fails to be an $A$-commute \iff\
both trips fail to be via $A$. Since the two trips are independent, the probability that a random $x$-$y$ commute
is an $A$-commute is $1-(1-p_A)^2= p_A(2-p_A)$.  By a similar argument as above
we obtain $\Ex Y =: \Ex Y_{i}= \frac{1}{p_A(2-p_A)}$ and so $\Ex T =: \Ecminxy = 2m  \effrxy \frac{R_A}{\effrxy}
\frac{1}{2-p_A} = 2m  R_A  \frac{1}{2-p_A} =  2m R_A \frac{1}{2-  R_B/({R_A + R_B})} = 2m R_A \frac{R_A + R_B}{2R_A + R_B}$.

Finally, to determine $\Ex Y$ in case \ref{iii} we argue as follows. To begin with, we have to make an expected $Y_{i}$
tries until we go via $A$  in at least one of the trips of some $x$-$y$ commute, i.e.\ until we achieve our first $A$-commute.
Then, unless our first $A$-commute was an \andcom{A}, we will have to make another $Y_{ii}$ tries in expectation to go via
$A$ in the other direction. By elementary calculations, an $A$-commute fails to be an \andcom{A}\ with probability
$q = \frac{2(1-p_A)}{(2-p_A)} =  2 \frac{R_A}{2R_A+R_B} $. Summing up, we have $\Ex Y =: \Ex Y_{iii} = \Ex Y_{i} + q \Ex Y_{ii}$.
Using our earlier calculations and \eqref{wald} this yields 
$\Ex T =: \Ecmaxxy = 2mR_A \left(\frac{R_A + R_B}{2R_A + R_B} + 2 \frac{R_A}{2R_A+R_B}  \right) = 2mR_A\frac{3R_A + R_B}{2R_A + R_B}$.

\end{proof}

\section{Main Results}

For a random walk on a network $N$ starting at a vertex $x$, define the random variable $\CR^E_x$, called the {\em edge cover-and-return time},
to be the first time when each edge of $N$ has been traversed and the particle is back at $x$.
Similarly, if $N$ is a digraph, we define $\CR^A_x$, called the {\em arc cover-and-return time}, to be the first time when each arc of $N$
has been traversed and the particle is back at $x$. Here, the directions of the edges do not affect the behavior of the random walk; the particle is allowed to traverse arcs backwards and its transition probabilities are always given by \eqref{trpr}. 

If $N$ is undirected we interpret it as a digraph by replacing each edge (loops included)
by an arc in each direction; thus, in that case, $\CR^A_x$ is the first time when each edge of $N$ has been traversed in both directions and
the particle is back at $x$.

\begin{theorem} \label{main}
Let $N=(G,\ell)$ be an undirected network and let $m:= \sum_{e\in E(G)} \ell(e)$. Then
\begin{enumerate}
\item \label{ci} $\Ex \CR^E(N) \leq 2m^2$ and
\item \label{cii} $\Ex \CR^A(N) \leq 3m^2$.
\end{enumerate}
\noindent
Moreover, if ${\overrightarrow N}$ is the result of orienting the edges of $N$ in any way, then
\begin{enumerate}
\addtocounter{enumi}{2}
\item \label{ciii} $\Ex \CR^A({\overrightarrow N}) \leq 2m^2$.
\end{enumerate}
\end{theorem}

\begin{proof}
Our proof follows the lines of the `spanning tree argument' used in \cite{AF} for the vertex cover time, the main difference being
that we apply \Lr{thcom} instead of \eqref{2lr}.

Let \sig\ be a closed walk in $G$ starting at an arbitrary fixed root $r$ and traversing each edge of $G$ precisely once in each direction.
Such a walk always exists: if $G$ is a tree then a depth-first search will do, and if not then one can consider a spanning tree $T$ of $G$,
construct such a walk on $T$, and then extend it to capture the chords. 

We are going to use \sig\ in order to define epochs for the random walk in $N$ starting at $x$, and then apply \Tr{thcom}
to bound the time time spent between pairs of such epochs. These epochs will differ depending on which of the
two cases we are considering.

{\em Cases \ref{ci} and \ref{ciii}:}  We prove the stronger result that $\Ex \CR^A_r({\overrightarrow N}) \leq 2m^2$.
let $e_i, i=1,2,\ldots, 2|E(G)|$, be the $i$th edge traversed by \sig,
with endvertices $x_i,y_i$ appearing in \sig\ in that order. For $i=1,2,\ldots, 2|E(G)|$ define the $i$th epoch
$\tau_i$ as follows.  If $e_i$ is directed from $x_i$ to $y_i$, then $\tau_i$ is the first time after $\tau_{i-1}$ when
\sig\ is at $y_i$ and has gone there from $x_i$ using the edge $e_i$ in that step,
where we set $\tau_0=0$. If $e_i$ is directed the other way, then we just let $\tau_i$ be the first time after
$\tau_{i-1}$ when \sig\ is at $y_i$.

Note that at time $\tau_{2|E(G)|}$ our random walk is back to $r$ and has performed an arc cover-and-return tour.
Thus $\Ex \CR^E_r({\overrightarrow N})$ is at most the expectation of $\tau_{2|E(G)|}$. Now the latter can be bounded using
\Tr{thcom} \ref{ii} as follows.  \Fe\ edge $e=xy$ of \G, there are precisely two time intervals bounded by the above epochs
corresponding to $e$: If $j,k$ are the two indices for which $e_j=e_k=e$, then these are the  time intervals $I_e^1:= [\tau_{j-1},\tau_j]$
and $I_e^2:= [\tau_{k-1},\tau_k]$. Note that, by the definition of $\tau_i$, the motion of the random walker
in the union of these two intervals is in fact an $\overrightarrow{e}$-commute or an $\overleftarrow{e}$-commute
(as defined in \Sr{seccom}) from $x$ to $y$, according to whether $e$ is directed from $x$ to $y$ or the other way round.
Thus, applying \Tr{thcom} \ref{ii} with $A=\{x,y;e\}$ and $B=G - \{e\}$ yields that the expected value $\ce_e$ of
$|I_e^1| +|I_e^2|$ is $2m \ell(e)$. Since the time interval $[0, \tau_{2|E(G)|}]$ is the union of all such pairs of intervals,
one pair for each $e\in E(G)$, we have
$$
\Ex \tau_{2|E(G)|} = \sum_{e\in E(G)} \ce_e = \sum_{e\in E(G)}2m \ell(e) = 2m^2.
$$
This yields $\Ex \CR^A_r({\overrightarrow N}) \leq 2m^2$, thus also $\Ex \CR^E_r(N) \leq 2m^2$, as claimed. 

{\em Case \ref{cii}:} To bound $\Ex \CR^A_r(N)$, we follow the same arguments as before, except that we define the
epochs slightly differently: the edges of \g are not directed now, and we always let $\tau_i$ be the first time
after $\tau_{i-1}$ when our random walk is at $y_i$ and has gone there from $x_i$ using the edge $e_i$ in that step.
We define the time intervals $I_e^1$ and $I_e^2$ as above, but this time we note that the motion of the random walker
in the union of these two intervals is an \andcom{e}. Applying \Tr{thcom} \ref{iii} with $A=\{x,y;e\}$ and $B=G - \{e\}$,
we obtain $\ce_e \leq 2m \ell(e) \frac{3 \ell(e) + R_B}{2 \ell(e) + R_B}$. This expression attains its maximum
value when $R_B=0$, and so we obtain $\ce_e \leq 2m \ell(e) \frac{3\ell(e) }{2\ell(e) } = 3m \ell(e)$.
Adding up all $\ce_e, e\in E(G)$ as above we conclude that $\Ex \CR^A_r(N) \leq 3m^2$, as claimed.
\end{proof}

The bounds of \Tr{main} are tight in the following situations.  For {\em Case \ref{ci}} and thus also {\em \ref{ciii}}, we have already noted
that a path of length $m$ takes time $2m^2$ to cover all edges and return.  For {\em Case \ref{cii}}, a network
consisting of a single vertex and a loop (of any length $m$) takes one step to cover that loop in one
direction, then on average two more to catch the other direction, so for this network $\Ex \CR^A = 3m^2$.

\section{Application to Brownian Motion and Infinite Networks}\label{secBM}

\subsection{Finite networks} \label{secBMf}

We begin this section by showing directly that (expected) edge-cover-and-return time---and, indeed, any kind of return time---is the
same in the Brownian model of random walk as it is in the $\ell^2$ model as defined in \Sr{secD}.  Fix a network $N$ and suppose that vertex $x$ of $N$ has incident edges
$e_1,\dots,e_k$ of lengths $\ell_1,\dots,\ell_k$ respectively.  (As usual loops must be represented twice in this list, once in each direction.)
The mean time taken by a walk in the $\ell^2$ model to traverse one of these edges, starting from $x$, is easily calculated using \eqref{trpr}:
$$
\sum_{i=1}^k  \frac{1}{\ell_i} \frac{1}{C_x} \ell_i^2 = \frac1{C_x} \sum_{i=1}^k \ell_i
$$
where $C_x := \sum_{j=1}^k (1/\ell_j)$.   The same is true in the Brownian model, because we may
identify the endpoints of the edges, calling the unified vertex $y$, and compute the commute time between $x$ and $y$
using the formula from Section~\ref{seccom}, which holds in both models.  By symmetry, the desired quantity is
half the commute time.

Since the expected time taken by a cover-and-return tour in either model is the sum over the vertices of $N$ of the expected time 
spent exiting those vertices, this quantity is the same for both models.  Note, however, that this does not mean that the
expected time taken in the Brownian model for a {\em particular} cover-and-return tour is the sum of the squares of the
lengths of its edges, and indeed that is not generally the case.

To compute the latter we need to know what the expected edge-traversal times are in the Brownian model.  This is not
needed to apply our bounds, but the computation is easy and as far as we know, has not appeared elsewhere.  We make
use of a couple of simple (and known, but proved here as well) facts about Brownian motion.

\begin{lemma}\label{lemma0} 
Consider standard Brownian motion on the positive real half axis, and let $t_b$ be the time of the first visit to $\ell\in \R^+$, and $t_a$ the time of the last visit to 0 before  $t_b$. 
Then the expected value  $T_\ell$ of $t_b-t_a$ is $\ell^2/3$.
\end{lemma}

\begin{proof} From scaling properties of Brownian motion (see, e.g., \cite{MP2}) we know that $T_\ell$ must be a multiple of $\ell^2$,
say $\alpha \ell^2$.  Then $T_{\ell/2}$ is $\alpha \ell^2/4$. Consider the first time $t_c$ that the particle reaches the point $\ell^2$. From there, it takes expected time $\ell^2/4$ to reach either 0 or $\ell$ for the first time again, reaching each of them first with equal probability $1/2$. Now conditioning on the event that $\ell$ was reached before 0 after $t_c$, we expect $t_b-t_a$ to be $T_{\ell/2} + \ell^2/4$, while if 0 was reached first then the experiment is effectively restarted, and we expect $t_b-t_a$ to be its overall expectation $T_\ell$. Combining we get $T_\ell = 1/2 (T_{\ell/2} + \ell^2/4) + 1/2 T_\ell$. Plugging in the above formulas for  $T_\ell$ and $T_{\ell/2}$ we can now solve for $\alpha$, and we get $\alpha = 1/3$.
\end{proof}

\newcommand{\T}{\mathcal{T}}
\begin{lemma}\label{lemmaT} Suppose a Brownian particle begins at vertex $x$ in a network $N$ and proceeds until it traverses one of the incident
edges $e_1,\dots,e_k$, and let $\T$ be the (random) time spent before the particle departs $x$ for the last time.  
Then $\T$ is independent of the index of the edge traversed.
\end{lemma}

\begin{proof} Since the past is irrelevant to the particle, there is a fixed distribution $\sigma$ on $\{1,2,\dots,k\}$ for
which edge is traversed after the particle departs $x$ for the last time (namely, the distribution whose probabilities are
proportional to the reciprocals of the edge-lengths).  Thus, the index of the traversed edge is independent of $\T$, and
necessarily, vice-versa.
\end{proof}

We are now ready to derive our formula.  Lemma~\ref{lemmaT} implies that in particular $\Ex\T$ is independent of which edge
is traversed.  Combining with Lemma~\ref{lemma0}, the expected time taken by the particle to traverse an edge from $x$,
given that it traversed edge $e_i$ first, is $\Ex\T + \ell_i^2/3$.  It follows that the expected time to traverse {\em some} edge
from $x$ is
$$
\sum_{i=1}^k \left( \Ex\T + \ell_i^2/3 \right) \frac{1}{\ell_i C_x} 
$$
which we know must be equal to $\frac1{C_x} \sum_{i=1}^k \ell_i$.  Solving gives $\Ex\T = \frac23 \frac{1}{C_x} \sum_{i=1}^k \ell_i$, and we have proved:

\begin{theorem} Suppose a standard Brownian particle on a network $N$ begins at vertex $x$ with incident edges $e_1,\dots,e_k$
of lengths $\ell_1,\dots,\ell_k$ respectively.  Then the expected time taken by the particle to traverse edges $e_i$, given
that it traversed $e_i$ first, is $\frac13 \ell_i^2 + \frac23 \sum_{j=1}^k \ell_j $.
\end{theorem}

A simple example of an edge-cover-and return tour with different expected times for the two models takes place on the network $N$ consisting
of two vertices $x$ and $y$, connected by an edge $e$ of length 1 and an edge $f$ of length 2.  Then the cover tour from $x$
consisting of $e$ then $f$ has expected time $5/3 + 8/3 = 13/3$ in the Brownian model, but constant time $1^2 + 2^2 = 9$
in the $\ell^2$ model.  The expected time for an edge-cover-and-return tour on $N$ is 7.7 in either model, although when
conditioned on (say) having started with edge $e$, the results are quite different.

\subsection{Infinite networks of finite total length} \label{secBMi}

We conclude with some implications for Brownian motion on infinite networks. 

As an illustrative example, consider the infinite binary tree $T$, with edge-lengths $4^{-k}$ at
level $k$ (the edges incident to the root counting as level 1). To this network of total length 1
it is natural to append a {\em boundary} $\partial T$: considered as a metric space, $T$ has a metric
completion $|T|$ and $\partial T$ is the set of completion points. An equivalent way to define
$\partial T$ is as the set of infinite paths starting at the root of $T$, which admits a
natural bijection to the set of infinite binary sequences. Note that $\partial T$ is homeomorphic to the Cantor set.

It is possible to prove that starting at the root of $T$, Brownian motion or random walk in the $\ell^2$ model
will almost surely reach `infinity', i.e.\ $\partial T$, after finite time.
Georgakopoulos and Kolesko \cite{GK} show that it is possible to let the particle continue its random motion afterwards:
they construct a random process on $|T|$ whose sample paths are continuous \wrt\ the topology of $|T|$ and behave like standard
Brownian motion in the neighborhood of each vertex of $T$. This construction was motivated by the results of \cite{G}
and an attempt to extend the theory of \cite{DS}, relating electrical networks and random processes, to the infinite case. 

Applied in this context, our results have a somewhat surprising implication: Brownian motion on $|T|$ will cover all edges
of $T$---and all of the continuum-many boundary points $\partial T$---in expected time at most 2, thus almost surely in finite time.
(We proved our results here for finite networks only, and so they cannot be directly applied to infinite ones.
However, the Brownian motion of \cite{GK} is constructed as a limit of the Brownian motions on an increasing sequence
of finite subgraphs of $T$, and as our results apply to each member of this sequence, they can be extended to the limit; see \cite{GK} for details.)

In fact, the Brownian motion of \cite{GK} is defined not only for trees, but also for arbitrary networks of finite total length.
It was shown in \cite{Gell} that every such network admits a boundary as above.  Moreover, this boundary, with its corresponding
topology, is well known and has had many applications.  To these results we have here added a (tight) bound on expected cover time,
perhaps adding further impetus to the study of networks of finite total length.

\section{Acknowledgments}

The portions of this research concerning Brownian motion benefited greatly from communications with Steve Evans (U.C. Berkeley) and with Yuval Peres and David B. Wilson (Microsoft Research).

\end{document}